\documentclass[12pt,oneside,reqno]{amsart}
\usepackage{amssymb}
\usepackage{amsmath}
\usepackage{mathtools}
\usepackage{mathrsfs}
\usepackage{amsthm}
\usepackage{amsfonts}
\usepackage{eucal}
\usepackage{dutchcal}
\usepackage[utf8]{inputenc}
\usepackage{comment}
\usepackage{marginnote}
\usepackage[style=alphabetic,backend = bibtex]{biblatex}
\usepackage[dvipsnames]{xcolor}
\usepackage{geometry}
\usepackage{hyperref}
\newtheorem{thm}{Theorem}
\newtheorem{prop}{Proposition}
\newtheorem{lem}{Lemma}

\numberwithin{thm}{section}
\numberwithin{prop}{section}
\numberwithin{lem}{section}
\numberwithin{equation}{section}
\title{EFFECTIVE RATIONAL APPROXIMATION ON SPHERES}
\date{}
\author{Zouhair Ouaggag}
\date{}

\begin{document}

\begin{abstract}
We prove an effective estimate for the counting function of Diophantine approximants on the sphere S$^n$. We use homogeneous dynamics on the space of orthogonal lattices, in particular effective equidistribution results and non-divergence estimates for the Siegel transform, developping on recent results of Alam-Ghosh and Kleinbock-Merrill.
\end{abstract}

\maketitle
\tableofcontents
\newpage

\section{Introduction and main result}

\paragraph{\textbf{Intrinsic Diophantine Approximation}}It is well-known in metric Diophantine approximation that for any $c>0$ and Lebesgue-almost all $\alpha \in \mathbb{R}^{m}$, there exist infinitely many solutions $(\textbf{p},q) \in \mathbb{Z}^{m}\times \mathbb{N}$ to the inequality\footnote{$|| \cdot||$ will denote the Euclidean norm.}
\begin{equation}
\label{diophantine}
 \left\lVert \alpha - \frac{\textbf{p}}{q}\right\rVert~ < ~  \frac{c}{q^{1+\frac{1}{m}}} \quad .
\end{equation}

A first refinement of this problem is to count solutions up to a certain bound for the complexity $q$ of the approximants, which leads to consider counting functions such as
$$
\mathsf{N}_{T,c}(\alpha) \coloneqq  \lvert \lbrace (\textbf{p},q) \in \mathbb{Z}^{m}\times \mathbb{N} : 1 \leq q < e^T \text{ and } (\ref{diophantine}) \text{ holds } \rbrace  \rvert~.
$$

Accurate estimates of the counting function $\mathsf{N}_{T,c}$ have already been established. We mention in particular the effective estimate by W. Schmidt \cite{Schmidt1960AMT}, who proved for more general approximating functions that for Lebesgue-almost $\alpha \in [0,1]^m$, 
\begin{equation}
\label{schmidt effective estimate}
\mathsf{N}_{T,c}(\alpha)=C_{c,m}\,T + O_{\alpha,\varepsilon}\left(  T^{1/2+\varepsilon}\right)~,
\end{equation}
for all $\varepsilon>0$, with a constant $C_{c,m}$ depending only on $c$ and $m$.\\

Another refinement of this problem is the so-called \emph{intrinsic} Diophantine approximation, where one considers a vector $\alpha$ in a manifold $X\subset \mathbb{R}^{m}$, for example level sets of a quadratic form, and is interested in counting rational approximants which also belong to $X$.
For rational approximation on spheres, Kleinbock and Merrill \cite{kleinbock2013rational} proved an analog of Dirichlet's theorem and a Kintchine-type dichotomy. For the divergence case in this last result, Alam and Ghosh \cite{alam2020quantitative} proved a quantitative estimate for the number of rational approximants on spheres with the critical Dirichlet exponent (Theorem \ref{quantitative approximation alam}). Our main result in this paper is to prove an estimate with error term analogous to (\ref{schmidt effective estimate}) for intrinsic Diophantine approximation on the sphere S$^n$.\\

Given $T, c>0$ and $\alpha \in$ S$^{n}$, we consider the following inequality (with the critical Dirichlet exponent for intrinsic Diophantine approximation on S$^{n}$)
\begin{equation}
\label{intrinsic diophantine}
 \left\lVert \alpha - \frac{\textbf{p}}{q}\right\rVert~ < ~  \frac{c}{q} \quad ,
\end{equation}
and the following counting function for intrinsic rational approximants
$$
N_{T,c}(\alpha) \coloneqq  \lvert \lbrace (\textbf{p},q) \in \mathbb{Z}^{n+1}\times \mathbb{N} : \frac{\textbf{p}}{q} \in \text{S}^n,~1 \leq q < \cosh T \text{ and } (\ref{intrinsic diophantine}) \text{ holds } \rbrace  \rvert~.
$$

We recall the following quantitative result of Alam and Ghosh.

\begin{thm}[\cite{alam2020quantitative}]
\label{quantitative approximation alam}
There exists a computable constant $C_{c,n}>0$, depending only on $c$ and $n$, such that for almost every $\alpha \in \emph{S}^{n}$,	
\begin{equation}
\label{quantitative approximation}
\frac{N_{T,c}(\alpha)}{T} \sim C_{c,n} \quad \text{as} ~ T \rightarrow \infty~.
\end{equation}
\end{thm}		
 
In the following theorem we give an estimate with error term improving (\ref{quantitative approximation}).
	
\begin{thm} 
\label{main theorem}
Let $n\geq 2$. There exist a computable constant $C_{c,n}>0$, depending only on $c$ and $n$, and a constant $\gamma <1$ depending only on $n$, such that for almost every $\alpha \in \emph{S}^{n}$, 
\begin{equation}
\label{effective approximation}
N_{T,c}(\alpha) = C_{c,n}T + O_{\alpha}(T^{\gamma})~.
\end{equation}
\end{thm}	

\paragraph{\textbf{Outline of the paper}}

We develop further the approach of \cite{kleinbock2013rational} and \cite{alam2020quantitative}, starting with the embedding of S$^n$ in the positive light cone $\mathcal{C}\coloneqq \{ x \in \mathbb{R}^{n+1}\times \mathbb{R}_+ :Q(x)=0\}$ for a quadratic form $Q$ of inertia $(n+1,1)$, and identifying good approximants $\frac{\textbf{p}}{q} \in$ S$^n$ with integer points $(\textbf{p},q)$ in $\mathbb{Z}^{n+2}\cap \mathcal{C}$ whose images under certain rotations $k \in K=\text{SO}(n+1)$ lie in a certain domain $E_{T,c} \subset \mathcal{C}$ (we recall more details about this correspondance in Section \ref{correspondance}). The number of solutions $N_{T,c}$ is then related to the number of lattice points in the domain $E_{T,c}$, which can be appoximated by a more convenient domain $F_{T,c}$ and tessellated by the action of a hyperbolic subgroup $\{a_t, t\in \mathbb{R} \} \subset $ SO$(Q)$. Estimating $N_{T,c}$ amounts then to analyzing ergodic averages of a counting function $\hat{\chi}$, namely the Siegel transform of the characterstic function of a fixed domain $F_{1,c}$, along $K$-orbits pushed by $\{a_t\}$.\vspace{7pt}\\
\indent In \cite{alam2020quantitative}, the authors use ergodicity of the $\{a_t\}$-action and Birkoff's Ergodic Theorem to obtain an asymptotic estimate for these ergodic averages. In order to obtain an estimate with error term, we use effective pointwise equidistribution along $\{a_t\}$-orbits, which we derive from effective double equidistribution of $K$-orbits pushed by $\{a_t\}$ (Proposition \ref{double equidistribution K orbits}), using a general method presented in \cite{kleinbock2017pointwise} to derive an "almost everywhere"-bound from an L$^p$-bound on ergodic averages (Proposition \ref{pointwise equidistribution}). However, using effective equidistribution requires to consider smooth and compactly supported test functions, whereas our counting function has typically none of these regularities. We address this issue in two steps.\vspace{7pt}\\ 
\indent We first introduce and study the integrability of the Siegel transform on the space of orthogonal lattices, then show (Proposition \ref{bounds truncated siegel transform}) that the Siegel transform $\hat{f}$ of a compactly supported function $f$ can be approximated by a truncated Siegel transform $\hat{f}^{(L)}$ in a way to control the approximation on translated $K$-orbits, i.e. control $| \hat{f}\circ a_t - \hat{f}^{(L)}\circ a_t|$ with respect to the probability measure on the orbits. In a second step (Proposition \ref{smooth approximation}), we show that the characteristic function $\chi$ of the elementary domain $F_{1,c}$ can be approximated by a family of smooth compactly supported functions $f_{\varepsilon}$, again in a way to keep control of the approximated Siegel transform $\hat{f}_{\varepsilon}$ over translated $K$-orbits. In this process we also need non-divergence results for the Siegel transform with respect to the probability measure on $K$-orbits (Propositions \ref{alpha_Lp} and \ref{non-espace of mass}). Thus, we can relate the counting function $ \hat{\chi}$ to a smooth and compactly supported function on the space of orthogonal lattices and then establish an effective almost-everywhere estimate for its ergodic averages (Theorem \ref{pointwise equidistribution counting function}). We finally derive an effective estimate for the number of solutions $N_{T,c}$ (Section \ref{proof of main theorem}).

\section{Diophantine approximation on S$^n$ and dynamics on the space of lattices}
\label{correspondance}

We recall the correspondence presented in \cite{kleinbock2013rational} and \cite{alam2020quantitative} between Diophantine approximation on the sphere S$^{n}$ and the dynamics of orthogonal lattices in $\mathbb{R}^{n+2} $.\\

We consider the quadratic form $Q: \mathbb{R}^{n+2} \rightarrow \mathbb{R}$ defined by
\begin{equation}
\label{quadratic form}
Q(x) := \sum_{i=1}^{n+1} x_i^{2} - x_{n+2}^{2}, \quad \text{ for } x= (x_1, \dots , x_{n+2})~,
\end{equation}
and the embedding of S$^{n}$ in the positive light cone
$$
\mathcal{C} := \{  x \in \mathbb{R}^{n+2} : Q(x)=0,~ x_{n+2} > 0 \}~,
$$
via $\alpha \mapsto (\alpha,1)$, which yields a one-to-one correspondance between primitive integer points on the positive light cone, $(\textbf{p},q) \in  \mathcal{C} \cap \mathbb{Z}^{n+2}_{\text{prim}}$, and rational points on the sphere, $ \frac{\textbf{p}}{q} \in S^{n}$.\\

We denote by $G$=SO$(Q)^\circ \cong$ SO$(n+1,1)^\circ$ the connected component of the group  of orientation-preserving linear transformations which preserve $Q$. We denote by $\Lambda_0 := \mathcal{C} \cap \mathbb{Z}^{n+2}$ the set of integer points on the positive light cone. By a \emph{lattice $\Lambda$ in $\mathcal{C}$} we mean a set of the form $g\Lambda_0$ for some $g\in G$. 
If we denote by $\Gamma$ the stabilizer of $\Lambda_0$ in $G$, then $\Gamma$ is a lattice in $G$ containing the subgroup SO$(Q)_{\mathbb{Z}}^\circ$ of integer points in $G$, as a finite index subgroup.
The space of lattices in $\mathcal{C}$ can be identified with the homogeneous space $\mathcal{X} := G/\Gamma$, endowed with the $G$-invariant probability measure $\mu_{\mathcal{X}}$.\\

Let $K$ denote the subgroup of $G$ that preserves the last coordinate in $\mathbb{R}^{n+2} $, i.e.
$$
K= \begin{pmatrix}
\text{SO}(n+1) &  \\
 & 1
\end{pmatrix} \cong \text{SO}(n+1)~,
$$
equipped with the Haar probability measure $\mu_K$.\\

The sphere S$^n$ can be realized as a quotient of $K$, endowed with a unique left $K$-invariant probability measure, giving a natural correspondance between full-measure sets in $K$ and those in S$^n$. \\

Let $\alpha\in$ S$^n$. For $k \in K$ such that $k(\alpha,1) = (0,\dots,0,1,1) \in \mathcal{C}$, and $(\textbf{p},q) \in \Lambda_0$, we write $k(\textbf{p},q)= (x_1,x_2, \dots , x_{n+2}) \in \mathcal{C}$, with $x_{n+2}=q$, and observe the following correspondance (\cite{alam2020quantitative}, Lemma 2.2.):
\begin{align*}
\left\lVert \alpha - \frac{\textbf{p}}{q}\right\rVert < \frac{c}{q} \quad &\Leftrightarrow \quad \lVert q( \alpha,1) -(\textbf{p},q) \rVert < c, \\
&\Leftrightarrow \quad \lVert qk(\alpha,1) - k(\textbf{p},q)\rVert < c, \\
&\Leftrightarrow \quad \lVert(x_1,x_2, \dots , x_n, x_{n+1}-x_{n+2},0)\rVert < c, \\
&\Leftrightarrow \quad 2x_{n+2}(x_{n+2}-x_{n+1}) < c^{2} \quad (\text{since }k(\textbf{p},q) \in \mathcal{C}). 
\end{align*}
Hence, if we denote 
$$E_{T,c}\coloneqq  \{ x \in \mathcal{C} : 2x_{n+2}(x_{n+2}-x_{n+1}) <c^{2}, 1\leq x_{n+2} < \cosh T \},
$$
then, for $k \in K$ such that $k(\alpha,1)=(0,\dots,0,1,1) \in \mathcal{C}$, we have:
\begin{equation}
\label{N_T,c as volume of E_T,c}
N_{T,c}(\alpha) =\lvert E_{T,c} \cap k\Lambda_0 \lvert.
\end{equation}

We denote $
\mathcal{Y} \coloneqq K \Lambda_0$, equipped with the Haar probability measure $\mu_{\mathcal{Y}}$.\\

We also consider elements 
$$
a_t =  \begin{pmatrix}
I_n & & \\
 &  \cosh t & -\sinh t \\
 & -\sinh t & \cosh t \end{pmatrix} \quad \in G
$$
and the corresponding one-parameter subgroup
$$
A = \left\lbrace  a_t : t \in \mathbb{R}\right\rbrace
$$
endowed with the natural measure $dt$.\\

We will denote by $\chi_E$ the characteristic function of a given set $E$ and use the notation $a \asymp b$ (resp. $a \ll b$) when there exist positive constants $C_1$ and $C_2$ such that $C_1b \leq a \leq C_2 b$ (resp. $a \leq C_2 b$).\\

In order to use the dynamics of translates of $\mathcal{Y}$ for the Diophantine approximation problem (\ref{N_T,c as volume of E_T,c}), we first approximate $E_{T,c}$ by a domain offering a convenient tessellation under the action of the subgroup $A$. We recall briefly the approach as in \cite{alam2020quantitative}.\\

\paragraph{\textbf{Appoximation of $E_{T,c}$}}
We approximate $E_{T,c}$ by the domain $F_{T,c}$ defined by 
$$
F_{T,c} := \{ x \in \mathcal{C} : x_{n+2}^{2}- x_{n+1}^{2} < c^{2}, 1\leq x_{n+2}+ x_{n+1} < e^{T} \}.
$$

We have indeed the following ''sandwiching'' of $E_{T,c}$ (see §3 in \cite{alam2020quantitative}), 
\begin{equation}
\label{sandwiching}
F_{T-r_0,c_\mathcal{l}} \setminus C_{\mathcal{l}} \quad \subseteq \quad
E_{T,c} \setminus C_0 \quad \subseteq \quad F_{T+r_0,c}~,
\end{equation} 
for large enough integers $\mathcal{l}$, with the sequence $c_\mathcal{l}\coloneqq c\cdot \left( 1-\frac{c^2}{2\mathcal{l}}\right)^{1/2}$, $c_\mathcal{l} \uparrow c$, a fixed constant $r_0>0$, and $K$-invariant sets 
$$C_0 \coloneqq \{ x \in \mathcal{C}: x_{n+2} \leq c^2+1\}, \text{ and } C_\mathcal{l} \coloneqq \{ x \in \mathcal{C}: x_{n+2} \leq \mathcal{l} \}~.$$ 

Then (\ref{N_T,c as volume of E_T,c}) and (\ref{sandwiching}) imply, for all $T>0$ and all $k \in K$ as in (\ref{N_T,c as volume of E_T,c}),
\begin{equation}
\label{sandwiching N_T,c}
|(F_{T-r_0,c_\mathcal{l}} \setminus C_{\mathcal{l}})\cap k\Lambda_0| ~\leq ~N_{T,c}(\alpha) + O(1) ~\leq~ |F_{T+r_0,c} \cap k\Lambda_0|. \\
\end{equation}

We observe further that 
\begin{equation*}
F_{1,c}\setminus F_{1,c_\mathcal{l}}= \{ x \in \mathcal{C} : c_\mathcal{l} \leq (x_{1}^{2}+ \dots + x_{n}^{2})^{1/2} < c, 1< x_{n+2}+ x_{n+1} < e \},
\end{equation*}
hence, as $l\rightarrow \infty$,
\begin{equation}
\label{volum F_1,c approximation}
\text{vol}(F_{1,c}) = \text{vol}( F_{1,c_\mathcal{l}})+ O\left( c-c_\mathcal{l} \right)=\text{vol}( F_{1,c_\mathcal{l}})+ O\left(\mathcal{l}^{-1} \right).
\end{equation}
\vspace{5pt}
\paragraph{\textbf{Tessellation of $F_{T,c}$}} We observe further that the domain $F_{T,c}$ can be tessellated using translates of the set $F_{1,c}$ under the action of $\{a_t\}$. We have, for all $N \geq 1$,
\begin{equation}
\label{tessaltion of F_T,c}
F_{N,c}= \bigsqcup_{j=0}^{N-1} a_{-j}(F_{1,c}).
\end{equation}

We denote by $\chi_{1,c}$ the characteristic function of $F_{1,c}$, and $\hat{\chi}_{1,c}$ its Siegel transform defined by
$$\hat{\chi}_{1,c}(\Lambda) \coloneqq \sum_{z \in \Lambda \setminus \{0\}} \chi_{1,c} (z),~ \text{ for all } \Lambda \in \mathcal{X}.$$

The tessellation (\ref{tessaltion of F_T,c}) implies, for all $T>0$ and all $\Lambda \in \mathcal{X}$,
\begin{equation}
\label{sandwiching F_T,c}
\sum_{t=0}^{\left\lfloor T \right\rfloor-1} \hat{\chi}_{1,c}(a_t\Lambda) ~\leq~ |F_{T,c} \cap \Lambda| ~\leq~ \sum_{t=0}^{\left\lfloor T \right\rfloor} \hat{\chi}_{1,c}(a_t\Lambda).
\end{equation}

We estimate the number of lattice points in $C_\mathcal{l}$ by $O\left( \mathcal{l}^{n} \right)$. It follows from (\ref{sandwiching N_T,c}) and (\ref{sandwiching F_T,c}), for all $T>0$ and $k \in K$ as in (\ref{N_T,c as volume of E_T,c}),
\begin{align}
\label{resandwiching N_T,c}
\sum_{t=0}^{\left\lfloor T-r_0 \right\rfloor-1} \hat{\chi}_{1,c_\mathcal{l}}(a_t k\Lambda_0) +O\left( \mathcal{l}^{n}\right) ~\leq ~N_{T,c}(\alpha) +O(1) ~\leq~ \sum_{t=0}^{\left\lfloor T+r_0 \right\rfloor} \hat{\chi}_{1,c}(a_t k\Lambda_0).
\end{align}

Thus, estimating $N_{T,c}(\alpha)$ amounts to analyzing ergodic sums of the form $\sum_{t=0}^{N} \hat{\chi}_{1,c}\circ a_t$ on $\mathcal{Y}= K\Lambda_0$. We will use for this purpose effective equidistribution results for unimodular lattices and specialize them to $\mathcal{Y}$. It will be important in our argument later that the error term in the effective equidistribution is explicit in terms of the $C^l$-norm, for some $l\geq 1$, of the test functions on $\mathcal{X}$. We introduce below the required notations.\\

Every $Y \in$ Lie$(G)$ defines a first order differential operator $D_Y$ on $C_c^{\infty}(\mathcal{X})$ by
$$
D_Y(\phi)(x) \coloneqq \frac{d}{dt}\phi(\exp (tY)x)|_{t=0}.
$$

If $\{ Y_1,\dots, Y_r\}$ is a basis of Lie$(G)$, then every monomial $Z=Y_1^{l_1}\dots Y_f^{l_r}$ defines a differential operator by 
\begin{equation}
\label{diff operator}
D_Z\coloneqq D_{Y_1}^{l_1}\dots D_{Y_r}^{l_r},
\end{equation}
of degree deg$(Z)=l_1+\dots+l_r$. For $l\geq 1$ and $\phi \in C_c^{\infty}(\mathcal{X})$, we write
$$
||\phi||_l \coloneqq ||\phi||_{C^l}= \sum_{\text{deg}(Z)\leq l} ||D_Z(\phi)||_{\infty}
$$

We recall in the following section some properties of the Siegel transform that we use later to analyze the ergodic sums $\sum_{t=0}^{N} \hat{\chi}_{1,c}\circ a_t$.

\section{Approximation of the counting function}

\subsection{Siegel transform}
\label{siegel transform}

Given a bounded measurable function $f:\mathbb{R}^{n+2} \rightarrow \mathbb{R}$ with compact support, its Siegel transform on the space $\mathcal{L}$ of unimodular lattices in $\mathbb{R}^{n+2}$ is defined by
\begin{equation}
\label{def siegel transform}
\hat{f}(\Lambda) \coloneqq \sum_{z \in \Lambda \setminus \{0\}} f(z), \quad \text{for } \Lambda \in \mathcal{L}.
\end{equation}

The Siegel transform of a bounded function is typically unbounded, but its growth rate is controlled by an explicit function $\alpha$ defined as follows.

Given a lattice $\Lambda \in \mathcal{L}$, we say that a subspace $V$ of $\mathbb{R}^{n+2}$ is $\Lambda$-\emph{rational} if the intersection $ V\cap\Lambda$ is a lattice in $V$. If $V$ is $\Lambda$-rational, we denote $d_\Lambda(V)$ the covolume of $V\cap \Lambda$ in $V$. We define then
$$
\alpha(\Lambda) \coloneqq \sup \left\lbrace d_{\Lambda}(V)^{-1}: V \text{ is a } \Lambda\text{-rational subspace of } \mathbb{R}^{n+2}  \right\rbrace.
$$ 

It follows from Mahler's Compactness Criterion that $\alpha$  is a proper map $\mathcal{L} \rightarrow [1, +\infty)$. We recall below some important properties.

\begin{prop}[\cite{Schmidt1968AsymptoticFF}]
\label{alpha_growth}
If $f:\mathbb{R}^{n+2} \rightarrow \mathbb{R}$ is a bounded function with compact support, then
$$|\hat{f}(\Lambda)| \ll_{supp(f)} ||f||_{\infty} \alpha(\Lambda), \quad \text{ for all } \Lambda \in \mathcal{L}.
$$
\end{prop}

We restrict this function to the space $\mathcal{X}$ of lattices on the positive light cone and denote it also by $\alpha$. Similarly to its L$^{p}$-integrability in $\mathcal{L}$ (see \cite{eskin1998upper}), we verify in the following proposition that $\alpha$ is also L$^{p}$-integrable in the space $\mathcal{X}$.

\begin{prop} 
\label{alpha_Lp}
The function $\alpha$ is in \emph{L}$^{p}(\mathcal{X})$ for $1\leq p < n$. In particular,
$$ \mu_{\mathcal{X}} (\{\alpha \geq L \}) \ll_p L^{-p},
\quad \text{ for all } p<n.$$
\end{prop}

\begin{proof}
We show $L^p$-integrability of the function $\alpha$ using reduction theory and an explicit measure of integration on $G$. For this purpose we consider a more convenient coordinate system and write
\begin{align*}
\mathrm{Q}(x)&=x_1x_{n+2}+x_2^2+\dots+x_{n+1}^2,~ \text{ for } x=(x_1,\dots,x_{n+2})\in \mathbb{R}^{n+2},\\
\mathrm{G}&=\text{SO}(\mathrm{Q})^\circ \cong \text{SO}(n+1,1)^\circ,\text{ and } \Gamma \text{ the stabilizer of }\Lambda_0 \text{ in }\mathrm{G},\\
\mathrm{A}&=\left\lbrace  \mathrm{a}_t:t\in \mathbb{R} \right\rbrace \subset G\quad \text{with} \quad \mathrm{a}_t:=\text{diag}(e^t,1,\dots,1,e^{-t}),\\
\mathrm{N} &= \left\lbrace  \mathrm{g}\in \mathrm{G} ~ : ~ \mathrm{a}_{-t}\mathrm{g}\mathrm{a}_{t}\rightarrow e \text{ as }t\rightarrow \infty \right\rbrace, \text{ the expanding horospherical subgroup to } \mathrm{A},\\
\mathrm{K}&\cong \text{SO}(n+1) \text{ a maximal compact in }\mathrm{G},
\end{align*}
with their respective Haar measures $\mu_\mathrm{G}$, $dt$, $\mu_\mathrm{N}$, $\mu_\mathrm{K}$, and consider the Iwasawa decomposition $\mathrm{G}=\mathrm{K}\mathrm{A}\mathrm{N}$.\\
We have $\mathrm{N}=\exp(\mathcal{N})$ with the space of positive roots
$$\mathcal{N} ~:=~ \oplus_{\rho(t) >1 } \{\mathcal{g} \in \text{Lie}(\mathrm{G}) : \mathrm{a}_t \mathcal{g}\mathrm{a}_{-t}= \rho(t) \mathcal{g}\}.$$
One verifies easily that the unique positive root is $\rho_0(t)=e^t$ with multiplicity $n$, which yields a measure of integration in the coordinates $\mathrm{G}=\mathrm{K}\mathrm{A}\mathrm{N}$ given by 
$$\mu_\mathrm{G}=\prod_{\rho(t)>1}\rho(t) d\mu_\mathrm{K} dtd\mu_\mathrm{N} = e^{nt}d\mu_\mathrm{K} dtd\mu_\mathrm{N}.$$
For $\beta \in \mathbb{R}$, let $\mathrm{A}_{\beta}\coloneqq \{\text{diag}(e^{t},1,\dots,1, e^{-t}): \: t < \beta \} \: \subset \: \mathrm{A}$. By reduction theory on the space of lattices $\mathcal{X}\cong \mathrm{G}/\Gamma$, we have that there exist $\beta >0$ and a compact set $\mathrm{N}_{0} \subset \mathrm{N}$ such that the union of finitely many translates of the Siegel set $\mathrm{S}_{\beta} := \mathrm{K}\mathrm{A}_{\beta}\mathrm{N}_0$ contains a fundamental domain for the action of $\Gamma$ on $\mathrm{G}$. Thus, it is enough to verify the integrability of $\alpha$ on $\mathrm{S}_{\beta}$. Moreover, since $\mathrm{N}_0$ and $\mathrm{K}$ are compact, $\{\mathrm{a}_t\mathrm{n}\mathrm{a}_{-t}: \mathrm{a}_t\in \mathrm{A}_{\beta}, \mathrm{n}\in \mathrm{N}_0\}$ uniformly bounded, and since there exists $C>0$ such that $\alpha(gx)\leq C\alpha(x)$ for all $x \in \mathcal{X}$ and uniformly for all $g$ in a compact set, we have for any $\mathrm{g}=\mathrm{k}\mathrm{a}_t\mathrm{n} \in \mathrm{S}_{\beta}$, 
$$\alpha(\mathrm{g}\mathbb{Z}^{n+2}) = \alpha(\mathrm{k}\mathrm{a}_t \mathrm{n}\mathrm{a}_{-t}\mathrm{a}_t\mathbb{Z}^{n+2}) \ll \alpha(\mathrm{a}_t\mathbb{Z}^{n+2}).$$ 
By definition of $\alpha$, we have further $\alpha(\mathrm{a}_t\mathbb{Z}^{n+2})= \max_{1\leq j \leq n+1} \prod_{1\leq i \leq j} \mathrm{a}_{t,i,j}^{-1}=e^{-t}$.
Hence,
\begin{align*}
&\int_{\mathcal{X}}\alpha(\Lambda)^{p} d\mu_{\mathcal{X}}(\Lambda) \ll \int_{\mathrm{S}_{\beta}} \alpha(\mathrm{g}\mathbb{Z}^{n+2})^{p}d\mu_\mathrm{G}(\mathrm{g} )\\
&\leq \int_{\mathrm{K}\mathrm{A}_{\beta}\mathrm{N}_0} \alpha(\mathrm{k}\mathrm{a}_t\mathrm{n}\mathbb{Z}^{n+2})^{p}e^{nt}d\mu_\mathrm{K}(\mathrm{k}) dtd\mu_\mathrm{N}(\mathrm{n}) \ll \int_{\mathrm{A}_{\beta}} \alpha(\mathrm{a}_t\mathbb{Z}^{n+2})^{p}e^{nt}dt  = \int_{-\infty}^{\beta} e^{-pt}e^{nt}dt ~ < \infty,
\end{align*}
for all $p < n $.\\
It follows in particular
$$
\mu_{\mathcal{X}}\left( \{ \alpha \geq L\}\right) \ll L^{-p}, \quad \text{for all } p < n.
$$
\end{proof}
Given a bounded measurable function $f:\mathcal{C}\rightarrow \mathbb{R}$ with compact support, we define and denote similarily to (\ref{def siegel transform}) its Siegel transform on the space $\mathcal{X}$,
\begin{equation}
\hat{f}(\Lambda) \coloneqq \sum_{z \in \Lambda \setminus \{0\}} f(z), \quad \text{for } \Lambda \in \mathcal{X}.
\end{equation}
We recall the Siegel Mean Value Theorem in the space of unimodular lattices $\mathcal{L}$ (see \cite{Siegel1945AMV}) and verify in the following proposition an analogous result for the space $\mathcal{X}$.

\begin{prop} 
\label{siegel mean value thrm}
If $f:\mathcal{C} \rightarrow \mathbb{R}$ is a bounded Riemann integrable function with compact support, then
$$ \int_{\mathcal{X}} \hat{f}(\Lambda) d\mu_{\mathcal{X}}(\Lambda) = \int_{\mathcal{C}} f(x) dx  $$
for some $G$-invariant measure $dx$ on $\mathcal{C}$.
\end{prop}
\begin{proof}
The map $f\mapsto \int_{\mathcal{X}}\hat{f}$ defines a positive $G$-invariant linear functional on the space of continuous compactly supported functions on $\mathcal{C}$, hence a $G$-invariant measure on $\mathcal{C}$. Uniqueness up to multiplication by a scalar of the invariant measure yields the claim.
\end{proof}
\subsection{Non-divergence estimates}
\label{subsection non-divergence}
In this subsection, we establish bounds for the Siegel transform $\hat{f}$ on translated $K$-orbits by analyzing the escape of mass on submanifolds $a_t\mathcal{Y} \subset \mathcal{X}$.

Following the same argument as in \cite{bjoerklund2018central} (Proposition 4.5) and using effective equidistribution of translated $K$-orbits (which we establish later in Proposition \ref{double equidistribution K orbits}) and L$^p$-integrability of the function $\alpha$ established in Proposition \ref{alpha_Lp}, we verify an analogous non-escape of mass in $a_t\mathcal{Y}$.

\begin{prop}
\label{non-espace of mass}
There exists $\kappa >0$ such that for every $L\geq 1$ and $t\geq \kappa \log L$, 
$$\mu_{\mathcal{Y}} (\{ y \in \mathcal{Y} : \alpha(a_t y) \geq L  \}) \ll_p L^{-p}, \quad \text{ for all } p<n.
$$
\end{prop}
\begin{proof}
Let $\chi_L$ be the characteristic function of the set $\{\alpha < L \} \subset \mathcal{X} $. By Mahler's Compactness Criterion, $\chi_L$ has compact support. Let $\rho \in C_c^{\infty}(G)$ be a non-negative function with $\int_G \rho =1$ and define 
$$
\eta_L(x) \coloneqq (\rho \ast \chi_L)(x) = \int_G \rho(g)\chi_L(g^{-1}x)d \mu_G(g), \quad \text{for all } x \in \mathcal{X}.
$$ 
Since $\mu_\mathcal{X}$ is $G$-invariant, we have 
$$
\int_\mathcal{X}\eta_L d\mu_\mathcal{X} = \int_\mathcal{X}\chi_Ld\mu_\mathcal{X}=\mu_\mathcal{X}(\{\alpha<L \}).
$$
Moreover, by invariance of $\mu_G$, we have for any differential operator $D_Z$ as in (\ref{diff operator}) that $D_Z\eta_L=D_Z(\rho)\ast \chi_L$, hence $\eta_L \in C_c^{\infty}(\mathcal{X})$ with $||\eta_L ||_{C^l}\ll ||\rho ||_{C^l} $.\\
Further, one can verify that there exists $c>1$ such that $\alpha(g^{-1}x)\geq c^{-1}\alpha(x)$ for every $g \in$ supp$(\rho)$ and $x \in \mathcal{X}$, hence $\{ \alpha\circ g^{-1} < L\} \subset \{\alpha < cL \} $ and $\eta_L \leq \chi_{cL} $. Thus,
\begin{align*}
\mu_{\mathcal{Y}}(\{y\in\mathcal{Y}:\alpha( a_ty) < cL \}) = \int_{\mathcal{Y}}\chi_{cL}(a_ty)d\mu_\mathcal{Y}(y) \geq \int_\mathcal{Y}\eta_L(a_ty)d\mu_\mathcal{Y}(y).
\end{align*}
By exponential equidistribution of translated $K$-orbits, which we prove later in Proposition \ref{double equidistribution K orbits}, there exist $\gamma>0$ and $l\geq 1$ such that
\begin{align*}
\int_\mathcal{Y}\eta_L(a_ty)d\mu_\mathcal{Y}(y) &= \int_\mathcal{X}\eta_L d\mu_\mathcal{X}+ O\left(e^{-\gamma t}||\eta_L||_{C^l} \right)\\
&=\mu_\mathcal{X}(\{\alpha<L \})+ O\left(e^{-\gamma t} 
\right),
\end{align*}
and by Proposition \ref{alpha_Lp},
$$
\mu_\mathcal{X}(\{\alpha\geq L \})\ll_p L^{-p} \quad \text{for all } p<n.
$$
Altogether we obtain
$$
\mu_\mathcal{Y}(\{y \in \mathcal{Y}: \alpha(a_ty)<cL \})\geq \mu_\mathcal{X}(\{\alpha<L \})+O\left(e^{-\gamma t} \right)=1+O_p\left(L^{-p}+e^{-\gamma t} \right),
$$
thus
$$
\mu_\mathcal{Y}(\{y\in\mathcal{Y}:\alpha(a_ty)\geq cL \}) \ll_p L^{-p}+e^{-\gamma t},
$$
which yields the claim for $s\geq \kappa\log L$ with $\kappa=\frac{p}{\gamma}$.
\end{proof}
An important estimate in our argument later is the integrability of the Siegel transform $\hat{f}$ on $a_t\mathcal{Y}$ uniformly in $t$. We use the following integrability estimate for the function $\alpha$. 

\begin{prop}[\cite{eskin1998upper}]
\label{bounds siegel transform}
If $n \geq 2$ and $0<s<2$, then for any lattice $\Lambda$ in $\mathbb{R}^{n+2}$
$$ \sup_{t> 0} \int_{K} \alpha ( a_t k \Lambda)^s d\mu_{K}(k) < \infty.$$
\end{prop}

\subsection{Truncated Siegel transform}
	The Siegel transform of a smooth compactly supported function is typically not bounded. To be able to apply equidistribution results, we truncate the Siegel transform using a smooth cut-off function $\eta_L$ built on the function $\alpha$. We use the same construction as in \autocite[Lemma 4.9]{bjoerklund2018central} which yields the following lemma.

\begin{lem}
For every $\xi>1$, there exists a family $(\eta_L)$ in C$_c^{\infty}(\mathcal{X})$ satisfying:
$$ 0 \leq \eta_L \leq 1 , \quad \eta_L = 1 \text{ on } \{ \alpha \leq \xi^{-1}L \} , \quad \eta_L = 0 \text{ on } \{ \alpha > \xi L \} , \quad || \eta_L||_{C^{l}} \ll 1.\\
$$
\end{lem}

For a bounded function $f:\mathcal{C} \rightarrow \mathbb{R}$ with compact support, we define the \emph{truncated Siegel transform} of $f$ by
$$
\hat{f}^{(L)} \coloneqq \hat{f}\cdot \eta_L.
$$

We record in the following proposition some properties of the truncated Siegel transform $\hat{f}^{(L)}$ which we use later in our arguments.  

\begin{prop}
\label{bounds truncated siegel transform}
For a bounded measurable function $f:\mathcal{C} \rightarrow \mathbb{R}$ with compact support, the truncated Siegel transform $\hat{f}^{(L)}$ satisfies the following bounds:
\begin{align}
& ||\hat{f}^{(L)} ||_{\infty} \ll_{\emph{supp}(f)} L||f ||_{\infty},\\
& \sup_{t\geq 0}||\hat{f}^{(L)}\circ a_t ||_{L^{p}_\mathcal{Y}} < \infty ~, ~~ \text{for all } 1\leq p<2,\\
& ||\hat{f} - \hat{f}^{(L)} ||_{L^{1}_\mathcal{X}} \ll_{\emph{supp}(f),\tau} L^{-(\tau-1)}||f ||_{\infty} \quad \text{, for all } \tau < n,\\
&||\hat{f}\circ a_t - \hat{f}^{(L)}\circ a_t ||_{L^{p}_\mathcal{Y}} \ll_{\emph{supp}(f),\tau} L^{-\frac{\tau(2-p)}{2p}}||f ||_{\infty} ~, ~~ \text{for all } 1\leq p<2, ~\tau<n \text{ and } t\geq \kappa \log L.
\end{align}
Moreover, if $f \in C_c^{\infty}(\mathcal{C})$ then $\hat{f}^{(L)} \in C_c^{\infty}(\mathcal{X})$ and satisfies
\begin{align}
\label{bound for Cl norm truncated}
||\hat{f}^{(L)} ||_{C^l} \ll_{\emph{supp}(f)} L||f ||_{C^l} \quad \text{, for all } l\geq 1.
\end{align}
\end{prop}

\begin{proof}
Since supp$(\eta_L) \subset \{\alpha \leq \xi L \}$, the first estimate follows from Proposition \ref{alpha_growth}.\\
Since $0\leq \eta_L\leq 1$, the second estimate follows immediately from Proposition \ref{bounds siegel transform}.\\
Further, since $\eta_L=1$ on $\{ \alpha \leq \xi^{-1}L\}$, it follows from Proposition \ref{alpha_growth} that
$$
\lVert \hat{f}-\hat{f}^{(L)} \rVert_{L^1_\mathcal{X}} = \int_{\mathcal{X}}|\hat{f}||1-\eta_L|d\mu_{\mathcal{X}} \ll_{\emph{supp}(f)} \int_{\{\alpha \geq \xi^{-1}L \}} \alpha d\mu_{\mathcal{X}}~||f ||_{\infty}.
$$
We apply Hölder's Inequality with $1 \leq p < n$ and $q=(1-1/p)^{-1}$ and deduce
$$
\lVert \hat{f}-\hat{f}^{(L)} \rVert_{L^1_\mathcal{X}} \ll_{\emph{supp}(f)}  ||\alpha ||_{L_{\mathcal{X}}^p} ~ \mu_{\mathcal{X}}(\{\alpha \geq \xi^{-1}L \})^{1/q}~||f ||_{\infty}.
$$
Then Proposition \ref{alpha_Lp} implies
$$
\lVert \hat{f}-\hat{f}^{(L)} \rVert_{L^1_\mathcal{X}} \ll_{\emph{supp}(f),p}  L^{-(p-1)}~||f ||_{\infty}.
$$
Similarly, Hölder's Inequality with $1\leq p<s<2$ and $q=(1/p-1/s)^{-1}$ gives
$$
\lVert \hat{f}\circ a_t-\hat{f}^{(L)}\circ a_t \rVert_{L^{p}_\mathcal{Y}} \ll_{\emph{supp}(f)}  ||\alpha \circ a_t ||_{L_{\mathcal{Y}}^s} ~ \mu_{\mathcal{Y}}(\{\alpha\circ a_t \geq \xi^{-1}L \})^{1/q}||f ||_{\infty},
$$
then Propositions \ref{non-espace of mass} and \ref{bounds siegel transform} imply, for all $t\geq \kappa \log L$ and $\varepsilon>0$,
\begin{align*}
\lVert \hat{f}\circ a_t-\hat{f}^{(L)}\circ a_t \rVert_{L^{p}_\mathcal{Y}} &\ll_{\emph{supp}(f),s,\varepsilon}  L^{-(n-\varepsilon)\frac{s-p}{sp}}||f ||_{\infty}\\
&\ll_{\emph{supp}(f),\tau} L^{\frac{-\tau (2-p)}{2p}}||f ||_{\infty}~, \text{ for all } \tau<n.
\end{align*}
For $f \in C_c^{\infty}(\mathcal{C})$ and any differential operator $D_Z$ as in (\ref{diff operator}), we observe that $D_Z(\hat{f})=\widehat{D_Z(f)}$. Hence, Proposition \ref{alpha_growth} implies
$$ |D_Z(\hat{f})| \ll_{\text{supp}(f)} || f||_{C^l} ~\alpha.
$$
Since supp$(\eta_L) \subset \{\alpha \leq \xi L \}$ and $||\eta_L ||_{C^l} \ll 1$, it follows 
$$ ||\hat{f}^{(L)} ||_{C^l} \ll_{\text{supp}(f)} L ||f ||_{C^l}.
$$
\end{proof}

\subsection{Smooth approximation of the counting function}
	For simplicity we write $\chi\coloneqq \chi_{F_{1,c}}$ for the characteristic function of the set $F_{1,c}$. We approximate $\chi$ by a family of non-negative functions $f_{\varepsilon} \in C_{c}^{\infty}(\mathcal{C})$ with support in an $\varepsilon$-neighborhood of $F_{1,c}$ such that
\begin{equation}
\label{epsilon approximation of chi}
\chi \leq f_{\varepsilon} \leq 1, \quad ||f_{\varepsilon} - \chi||_{\text{L}^{1}_\mathcal{C}} \ll \varepsilon , \quad ||f_{\varepsilon} - \chi||_{\text{L}^{2}_\mathcal{C}} \ll \varepsilon^{1/2}, \quad ||f_{\varepsilon}||_{C^{l}} \ll \varepsilon^{-l}.
\end{equation}

The following proposition shows that this approximation of $\chi$ also yields a good approximation of its Siegel transform $\hat{\chi}$ on translated $K$-orbits in the following sense.
\begin{prop}
\label{smooth approximation}
There exists $\theta>0$ such that for every $\varepsilon>0$ and $t\geq -\frac{1}{\theta}\log \varepsilon$,
$$\int_{\mathcal{Y}} |\hat{f_{\varepsilon}} \circ a_t - \hat{\chi}\circ a_t| d\mu_{\mathcal{Y}} \ll \varepsilon
$$
\end{prop}

\begin{proof}
We first observe that there exists $c_{\varepsilon}>c$ such that $c_{\varepsilon}=c+ O(\varepsilon)$ and $f_{\varepsilon}\leq \chi_{\varepsilon}$, where $\chi_{\varepsilon}$ denotes the characteristic function of the set
$$
\left\lbrace x \in \mathcal{C} ~:~  1-\varepsilon \leq x_{n+2}+x_{n+1}\leq e+\varepsilon ,~   x_{n+2}^2-x_{n+1}^2 < c_{\varepsilon}^2    \right\rbrace.
$$
The difference $\chi_{\varepsilon}-\chi$ is bounded by the sum $\chi^{(1)}_{\varepsilon}+\chi^{(2)}_{\varepsilon}+\chi^{(3)}_{\varepsilon}$ of the characteristic functions of the sets
\begin{align*}
&\left\lbrace x \in \mathcal{C} ~:~ 1-\varepsilon \leq x_{n+2}+x_{n+1}\leq 1 ,~   x_{n+2}^2-x_{n+1}^2 < c_{\varepsilon}^2    \right\rbrace,\\
&\left\lbrace x \in \mathcal{C} ~:~ e \leq x_{n+2}+x_{n+1}\leq e+\varepsilon,~   x_{n+2}^2-x_{n+1}^2 < c_{\varepsilon}^2    \right\rbrace,\\
&\left\lbrace x \in \mathcal{C} ~:~ 1 \leq x_{n+2}+x_{n+1}\leq e ,~   c^2<x_{n+2}^2-x_{n+1}^2 < c_{\varepsilon}^2    \right\rbrace.
\end{align*} 
Since $0\leq \chi \leq f_{\varepsilon} \leq \chi_{\varepsilon}$, it follows in particular 
$$
\hat{f_{\varepsilon}} (a_t\Lambda) - \hat{\chi}(a_t\Lambda) \leq \hat{\chi}^{(1)}_{\varepsilon}(a_t\Lambda)+\hat{\chi}^{(2)}_{\varepsilon}(a_t\Lambda)+\hat{\chi}^{(3)}_{\varepsilon}(a_t\Lambda).
$$
We first consider $\chi^{(1)}_{\varepsilon}$. For $x$ in the corresponding set, we also have
\begin{equation*}
0\leq x_{n+2}-x_{n+1} < c_\varepsilon^2/(1-\varepsilon) \quad \text{and}\quad x_1^2+ \dots +x_{n}^2 < c_\varepsilon^2.\\
\end{equation*}
We write $I_{0,\varepsilon}\coloneqq [0,c_\varepsilon] $, $I_{1,\varepsilon}\coloneqq [-c_{\varepsilon}^2/(1-\varepsilon),0]$, $I_{2,\varepsilon}\coloneqq [1-\varepsilon,1] $, $ k=(k_1, \dots , k_{n+2})^T \in K$, and compute
\begin{align*}
&\int_{\mathcal{Y}} |\hat{\chi}^{(1)}_{\varepsilon}\circ a_t| ~d\mu_{\mathcal{Y}} = \int_{K} \hat{\chi}^{(1)}_{\varepsilon}(a_tk\Lambda_0) ~d\mu_K(k) = \int_{K}\sum_{z \in \Lambda_0} \chi_{\varepsilon}^{(1)}(a_tkz) ~d\mu_K(k) \\
&=   \sum_{z \in \Lambda_0 }\int_{K} \chi^{(1)}_{\varepsilon}\left(\begin{matrix}
 \langle k_{1},z \rangle\\ \dots\\  \langle k_{n},z \rangle\\
 \langle k_{n+1},z \rangle\cosh t - z_{n+2}\sinh t\\
 \langle k_{n+1},z \rangle(-\sinh t) + z_{n+2}\cosh t 
\end{matrix}\right)~d\mu_K(k) \\
&=\sum_{z \in \Lambda_0} 
\int_K 
\chi_{I_{0,\varepsilon}}\left( ||\langle k_1,z\rangle ,\dots, \langle k_n,z \rangle || \right) 
\chi_{I_{1,\varepsilon}}\left(e^{t}\left(\langle k_{n+1},z \rangle -z_{n+2}\right) \right)
\chi_{I_{2,\varepsilon}}\left(e^{-t}\left( \langle k_{n+1},z \rangle +z_{n+2}\right) \right)
d\mu_K(k).
\end{align*}
We observe that the intersection $(e^{-t}I_{1,\varepsilon}+z_{n+2}) \cap (e^{t}I_{2,\varepsilon}-z_{n+2})$ is non-empty only if $(1-\varepsilon)e^t\leq 2z_{n+2} \leq e^t+\frac{c_\varepsilon^2}{1-\varepsilon}e^{-t}$, i.e. $z_{n+2}=e^t/2 +O(\varepsilon e^t+e^{-t})$. Moreover, writing each $z \in \Lambda_0$ as $z = z_{n+2}k_zv_0$ with some $k_z \in K$ and $v_0=(0,\dots,0,1,1)\in \mathcal{C}$, and using invariance under $k_z$, we have
\begin{align}
&\int_{\mathcal{Y}} |\hat{\chi}^{(1)}_{\varepsilon}\circ a_t| ~d\mu_{\mathcal{Y}}\nonumber\\
&\leq \sum_{\substack{z\in \Lambda_0\\ z_{n+2}=\frac{e^t}{2} +O(\varepsilon e^t+e^{-t})}} \int_K 
 \chi_{I_{0,\varepsilon}}\left( z_{n+2}||\langle k_1,v_0\rangle ,\dots, \langle k_n,v_0 \rangle || \right)
\chi_{e^{-t}I_{1,\varepsilon}}\left(z_{n+2}(\langle k_{n+1},v_0 \rangle -1) \right)\cdot \nonumber\\
&\qquad \qquad \qquad \qquad \qquad \qquad \qquad \qquad \qquad \qquad \qquad \qquad \cdot \chi_{e^{t}I_{2,\varepsilon}}\left(z_{n+2}(\langle k_{n+1},v_0 \rangle +1) \right)
d\mu_K(k) \nonumber \\
&\leq \sum_{\substack{z\in \Lambda_0\\ z_{n+2}=\frac{e^t}{2} +O(\varepsilon e^t+e^{-t})}} \int_K 
\chi_{e^{-t}\frac{2}{1-\varepsilon}I_{0,\varepsilon}}\left(||\langle k_1,v_0\rangle ,\dots, \langle k_n,v_0 \rangle || \right)
\chi_{e^{-2t}\frac{2}{1-\varepsilon} I_{1,\varepsilon}}\left(\langle k_{n+1},v_0\rangle - 1 \right) \cdot \nonumber\\
&\qquad \qquad \qquad \qquad \qquad \qquad \qquad \qquad \qquad \qquad \qquad \qquad \cdot \chi_{\frac{2}{1-\varepsilon}I_{2,\varepsilon}}\left(\langle k_{n+1},v_0 \rangle +1 \right)
d\mu_K(k) \nonumber \\
&\leq \sum_{\substack{z\in \Lambda_0\\ z_{n+2}=\frac{e^t}{2} +O(\varepsilon e^t+e^{-t})}} \mu_K \left( \left\lbrace k \in K ~ : \begin{matrix}  |k_{i,n+1}| \ll e^{-t}, ~ i=1,\dots,n  , \\
|k_{n+1,n+1}-1| \ll \min(e^{-2t},\varepsilon).
\end{matrix} \right\rbrace \right)\nonumber \\
&\leq \sum_{\substack{z\in \Lambda_0\\ z_{n+2}=\frac{e^t}{2} +O(\varepsilon e^t+e^{-t})}} \mu_K \left( \left\lbrace k \in K ~ : ||kv_0-v_0 ||\ll e^{-t} \right\rbrace \right)\nonumber \\
&\ll \sum_{\substack{z\in \Lambda_0\\ z_{n+2}=\frac{e^t}{2} +O(\varepsilon e^t+e^{-t})}} \mu_{\text{S}^n} \left( \left\lbrace v \in \text{S}^n ~ :  ||v-v_0 ||\ll e^{-t} \right\rbrace \right)\nonumber \\
&\ll \sum_{\substack{z\in \Lambda_0\\ z_{n+2}=\frac{e^t}{2} +O(\varepsilon e^t+e^{-t})}} e^{-nt}.\nonumber 
\end{align}
We use further that there exist positive constants $C$ and $\theta$ such that, for all $n\geq 2$, we have
\begin{equation*}
|\{z \in \mathcal{C}\cap \mathbb{Z}^{n+2}~: 0\leq z_{n+2}<T \}| = CT^{n}+O(T^{n-\theta}), 
\end{equation*}
hence 
\begin{equation*}
|\{z \in \mathcal{C}\cap \mathbb{Z}^{n+2}~: (1-\varepsilon)e^t \leq 2z_{n+2}<e^t+\frac{c_\varepsilon^2}{1-\varepsilon}e^{-t} \}| \leq \varepsilon e^{nt}+O(e^{(n-\theta)t}).
\end{equation*}
It follows
\begin{equation}
\int_{\mathcal{Y}} |\hat{\chi}^{(1)}_{\varepsilon}\circ a_t| ~d\mu_{\mathcal{Y}} \ll \varepsilon + e^{-\theta t}.
\end{equation}
\\
We proceed similarly for $\chi^{(3)}_{\varepsilon}$. For $x$ in the corresponding set, we also have
\begin{equation*}
\frac{c^2}{e}\leq x_{n+2}-x_{n+1} < c_\varepsilon^2 \quad \text{and}\quad c^2<x_1^2+ \dots +x_{n}^2 < c_\varepsilon^2.\\
\end{equation*}
We write $I'_{0,\varepsilon}\coloneqq [c,c_\varepsilon] $, $I'_{1,\varepsilon}\coloneqq [-c_{\varepsilon}^2,-c^2/e]$, $I'_{2}\coloneqq [1,e] $ and compute similarly
\begin{align*}
&\int_{\mathcal{Y}} |\hat{\chi}^{(3)}_{\varepsilon}\circ a_t| ~d\mu_{\mathcal{Y}}  = \int_{K}\sum_{z \in \Lambda_0} \chi^{(3)}_{\varepsilon}(a_tkz) ~d\mu_K(k) \\
&=\sum_{z \in \Lambda_0} 
\int_K 
\chi_{I'_{0,\varepsilon}}\left( ||\langle k_1,z\rangle ,\dots, \langle k_n,z \rangle || \right) 
\chi_{I'_{1,\varepsilon}}\left(e^{t}\left(\langle k_{n+1},z \rangle -z_{n+2}\right) \right)
\chi_{I'_{2,\varepsilon}}\left(e^{-t}\left( \langle k_{n+1},z \rangle +z_{n+2}\right) \right)
d\mu_K(k).
\end{align*}
We observe again that the intersection $(e^{-t}I'_{1,\varepsilon}+z_{n+2}) \cap (e^{t}I'_{2}-z_{n+2})$ is non-empty only if $C_1e^t \leq z_{n+2} \leq C_2e^t$ for some positive constants $C_1$ and $C_2$. Moreover, writing each $z \in \Lambda_0$ as $z = z_{n+2}k_zv_0$ with some $k_z \in K$ and $v_0=(0,\dots,0,1,1)\in \mathcal{C}$, and using invariance under $k_z$, we have
\begin{align}
&\int_{\mathcal{Y}} |\hat{\chi}^{(3)}_{\varepsilon}\circ a_t| ~d\mu_{\mathcal{Y}}\nonumber\\
&\leq \sum_{\substack{z\in \Lambda_0\\ z_{n+2} \asymp e^t}} \int_K 
 \chi_{I'_{0,\varepsilon}}\left( z_{n+2}||\langle k_1,v_0\rangle ,\dots, \langle k_n,v_0 \rangle || \right)
\chi_{e^{-t}I'_{1,\varepsilon}}\left(z_{n+2}(\langle k_{n+1},v_0 \rangle -1) \right)\cdot \nonumber\\
&\qquad \qquad \qquad \qquad \qquad \qquad \qquad \qquad \qquad \qquad \qquad \qquad \cdot \chi_{e^{t}I'_{2}}\left(z_{n+2}(\langle k_{n+1},v_0 \rangle +1) \right)
d\mu_K(k) \nonumber \\
&\leq \sum_{\substack{z\in \Lambda_0\\ z_{n+2} \asymp e^t}} \int_K 
\chi_{e^{-t}\frac{1}{C_1}I'_{0,\varepsilon}}\left(||\langle k_1,v_0\rangle ,\dots, \langle k_n,v_0 \rangle || \right)
\chi_{e^{-2t}\frac{1}{C_1}I'_{1,\varepsilon}}\left(\langle k_{n+1},v_0\rangle - 1 \right) \cdot \nonumber\\
&\qquad \qquad \qquad \qquad \qquad \qquad \qquad \qquad \qquad \qquad \qquad \qquad \cdot \chi_{\frac{1}{C_1}I'_{2}}\left(\langle k_{n+1},v_0 \rangle +1 \right)
d\mu_K(k) \nonumber \\
&\leq \sum_{\substack{z\in \Lambda_0\\ z_{n+2} \asymp e^t}} \mu_K \left( \left\lbrace k \in K ~ : ||kv_0-v_0 ||\ll \varepsilon
e^{-t} \right\rbrace \right)\nonumber \\
&\ll \sum_{\substack{z\in \Lambda_0\\ z_{n+2} \asymp e^t}} \mu_{\text{S}^n} \left( \left\lbrace v \in \text{S}^n ~ :  ||v-v_0 ||\ll \varepsilon e^{-t} \right\rbrace \right)\nonumber \\
&\ll \sum_{\substack{z\in \Lambda_0\\ z_{n+2} \asymp e^t}} \varepsilon^n e^{-nt}. \nonumber 
\end{align}
We use again the estimate
\begin{equation*}
|\{z \in \mathcal{C}\cap \mathbb{Z}^{n+2}~: z_{n+2}\asymp e^{t} \}|= O(e^{nt}), 
\end{equation*}
hence 
\begin{equation*}
\int_{\mathcal{Y}} |\hat{\chi}^{(3)}_{\varepsilon}\circ a_t| ~d\mu_{\mathcal{Y}} ~\ll~ \varepsilon.
\end{equation*}
The bound for $||\hat{\chi}^{(2)}_{\varepsilon}\circ a_t ||_{L^1_\mathcal{Y}}$ is obtained similarly as for $\chi^{(1)}_{\varepsilon}$.\\
Altogether we obtain, for all $t\geq-\frac{1}{\theta} \log \varepsilon$,
$$||\hat{f}_\varepsilon \circ a_t -\hat{\chi}\circ a_t ||_{L^1_\mathcal{Y}} \ll \varepsilon.
$$
\end{proof}

\section{Effective equidistribution}

	\subsection{Effective double equidistribution of translated $K$-orbits}
	
In this section we prove an effective equidistribution of $K$-orbits by relating it to effective equidistribution of unstable horospherical orbits established in a more general setting in \cite{bjoerklund2021equidistribution}. We recall the notations
\begin{align*}
G &= \text{SO}(Q)^\circ \cong \text{SO}(n+1,1)^\circ,\\
K &= \begin{pmatrix}
\text{SO}(n+1) &  \\
 & 1
\end{pmatrix},\\
a_t &=  \begin{pmatrix}
I_n & & \\
 &  \cosh t & -\sinh t \\
 & -\sinh t & \cosh t \end{pmatrix} ~ \in G, \quad \text{and }~ A = \left\lbrace  a_t : t \in \mathbb{R}\right\rbrace.
\end{align*}

We also consider the corresponding horospherical subgroups
\begin{align*}
&U = \left\lbrace  g\in G ~ : ~ a_{-t}ga_{t}\rightarrow e \text{ as }t\rightarrow \infty \right\rbrace,\\
&U^- = \left\lbrace  g\in G ~ : ~ a_{t}ga_{-t}\rightarrow e \text{ as }t\rightarrow \infty \right\rbrace,\\
&H = \left\lbrace  g\in G ~ : ~ a_{t}g=ga_{t}\right\rbrace,
\end{align*}
and the probability measures $d\mu_K$, $dt$ and $d\mu_U$ on $K$, $A$ and $U$ respectively.

We denote $B_r^K$ the ball of radius $r>0$ centered at the identity in $K$.
\begin{thm}[specializes Theorem 1.1. in \cite{bjoerklund2021equidistribution}]
\label{double equidistribution U orbits}
There exists $\delta>0$ and $l\geq 1$ such that, for every $f \in C_{c}^{\infty}(U)$ and $\varphi, \psi \in C_{c}^{\infty}(\mathcal{X})$ and every compact subset $L \subset \mathcal{X}$, there exists $C>0$ such that for every $\Lambda \in L$ and $t_1, t_2 \geq 0$, one has
$$
\left|  \int_{U} f(u)\varphi (a_{t_1} u\Lambda) \psi (a_{t_2} u\Lambda)d\mu_U(u) - \int_{U}f \int_{\mathcal{X}} \varphi \int_{\mathcal{X}} \psi \right| \leq C e^{-\delta \min (t_1,t_2, |t_1-t_2|)}||f||_l||\varphi||_l||\psi ||_l~.
$$ 
\end{thm}
In the following proposition we prove an analogous effective double equidistribution for translated $K$-orbits.
\begin{prop}
\label{double equidistribution K orbits}
There exists $\gamma>0$ and $l\geq 1$ such that, for every $f \in C_{c}^{\infty}(K)$ and $\varphi, \psi \in C_{c}^{\infty}(\mathcal{X})$ and every compact subset $L \subset \mathcal{X}$, there exists $C>0$ such that for every $\Lambda \in L$ and $t_1, t_2 \geq 0$, we have
$$
\left|  I_{\Lambda,f,\varphi,\psi}(t_1,t_2) - \int_{K}f \int_{\mathcal{X}} \varphi \int_{\mathcal{X}} \psi \right| \leq C e^{-\gamma \min (t_1,t_2, |t_1-t_2|)}||f||_l||\varphi||_l||\psi ||_l~,
$$ 
where $I_{\Lambda,f,\varphi,\psi}(t_1,t_2) \coloneqq \int_{K} f(k)\varphi (a_{t_1} k\Lambda) \psi (a_{t_2} k\Lambda)d\mu_K(k)$.
\end{prop}

\begin{proof}
We consider the centralizer of $A$ in $K$,  
\begin{equation*}
 M\coloneqq \text{cent}_K(A)=K\cap H = \begin{pmatrix}
\text{SO}(n) &  \\
 & I_2
\end{pmatrix} \cong \text{SO}(n),
\end{equation*}
and the submanifold $S\subset K$ defined via the exponential map by
$$
\text{Lie}(S)= \left\lbrace \begin{pmatrix}
0_n &\textbf{s}& &  \\
-\textbf{s}^T &0 &\\
& & 0
\end{pmatrix} : \textbf{s}\in \mathbb{R}^n\right\rbrace.$$
We have $\text{Lie}(K)=\text{Lie}(M)\oplus\text{Lie}(S)$ and the map $M\times S\rightarrow K$ is a diffeomorphism in a neighborhood of the identity, giving a unique decomposition $k= m(k)s(k)$ and also $d\mu_K=d\mu_M\times d\mu_S$, where we denote by $d\mu_M$ the Haar measure on $M$ and by $d\mu_S$ a smooth measure defined on a neighborhood of the identity in $S$.\\
Further, we consider the decomposition  of $G$ as the product $U^-HU$ in a neighborhood of the identity, giving a unique decomposition $s=u^-(s)h(s)u(s)$. We verify that the coordinate map $S\rightarrow U$, $s\mapsto u(s)$ is a diffeomorphism in a neighborhood of the identity. We first observe that
\begin{equation*}
\text{dim}(S)=\text{dim}(K)-\text{dim}(M)=\frac{(n+1)n}{2}-\frac{(n-1)n}{2}=n=\text{dim}(U).
\end{equation*}
Moreover, for the product map $p:U^-\times H \times U \rightarrow G,~ (u^-,h,u)\mapsto u^-hu$, the derivative at the identity is given by $D(p)_e(x,y,z)=x+y+z$, for all $(x,y,z)\in \text{Lie}(U^-) \times\text{Lie}(H) \times\text{Lie}(U)$. Hence, for all $w \in \text{Lie}(G)$, the $U$-component of $D(p)_e^{-1}(w)$ is zero if and only if $w \in \text{Lie}(U^-)+\text{Lie}(H)$. Since $\text{Lie}(S) \cap \left(\text{Lie}(U^-)+\text{Lie}(H)\right)= 0$, the derivative of $s\mapsto u(s)$ is injective.\\
We localize the problem to a neighborhood of the identity by considering the partition of unity $1= \sum_{i=1}^{N} \phi_i(kk_i^{-1})$ for some $k_i \in$ supp$(f)$ and all $k \in \text{supp}(f)$, with non-negative functions $\phi_i \in C_c^{\infty}(K)$ such that supp$(\phi_i) \subseteq B_r^K$, $|| \phi_i ||_l \ll r^{-\nu}$ and $N \ll r^{-\lambda}$, for some $\nu$, $\lambda >0$, and for $r>0$ small enough to be fixed later.\\
We write for simplicity $k=m_ks_k=m_ku_{s_k}^-h_{s_k}u_{s_k}$, the unique decompositions of $k$ and $s$ in a neighborhood of the identity in $K$ and $S$. We also write $f_i(k)\coloneqq f(kk_i)$ and $\Lambda_i \coloneqq k_i\Lambda$. We compute
\begin{align*}
&I_{\Lambda,f,\varphi,\psi}(t_1,t_2) = \sum_{i=1}^N \int_K \phi_i(k) f(kk_i)\varphi(a_{t_1}kk_i\Lambda)\psi(a_{t_2}kk_i\Lambda)d\mu_K(k)\\
&=\sum_{i=1}^N \int_K \phi_i(k) f_i(k)\varphi(m_ka_{t_1}u^-_{s_k}a_{-t_1}h_{s_k}a_{t_1}u_{s_k}\Lambda_i)\psi(m_ka_{t_2}u^-_{s_k}a_{-t_2}h_{s_k}a_{t_2}u_{s_k}\Lambda_i)d\mu_K(k).
\end{align*}
By Lipschitz continuity of the coordinate maps $m_k$, $u^{-}_{s_k}$ and $h_{s_k}$ on $B_r^K$ with $r$ small enough, there exists a constant $C_1>0$ such that for all $k \in B_r^K$, we have 
$$ a_tu^-_{s_k}a_{-t} \in B_{C_1re^{-2t}}^K~ ~~~ \text{and}~~~ m_k,h_{s_k} \in B_{C_1r}^K.$$
By Lipschitz continuity of $\varphi$ and $\phi$, it follows
$$
\left|~I_{\Lambda,f,\varphi,\psi}(t_1,t_2) - \sum_{i=1}^N \int_K \phi_i(k) f_i(k)\varphi(a_{t_1}u_{s_k}\Lambda_i)\psi(a_{t_2}u_{s_k}\Lambda_i)d\mu_K(k) ~\right|~\ll_l~ r ||f ||_l||\varphi ||_l ||\psi  ||_l.
$$
We use now the decomposition $d\mu_K=d\mu_M\times d\mu_S$ and apply the change of variable $u \mapsto s(u)=s_{u}$, with a density $\rho$ defined in a neighborhood of the identity in $U$ by
\begin{equation*}
\int_S \Phi(s)d\mu_S(s)=\int_U\Phi(s(u))\rho(u)d\mu_U(u) \quad \text{for all } \Phi \in C_c(S) \text{ with supp}(\Phi)\subset B_r^S. 
\end{equation*}
We have
\begin{align}
&\sum_{i=1}^N \int_K \phi_i(k) f_i(k)\varphi(a_{t_1}u_{s_k}\Lambda_i)\psi(a_{t_2}u_{s_k}\Lambda_i)d\mu_K(k) \nonumber \\
&=~\sum_{i=1}^N \int_{M\times S} \phi_i(ms) f_i(ms)\varphi(a_{t_1}u_{s}\Lambda_i)\psi(a_{t_2}u_{s}\Lambda_i)d\mu_S(s)d\mu_M(m) \nonumber\\
&= ~\int_M\left(\sum_{i=1}^N \int_U \phi_i(ms_u) f_i(ms_u)\varphi(a_{t_1}u\Lambda_i)\psi(a_{t_2}u\Lambda_i)\rho(u) d\mu_U(u)\right)d\mu_M(m). \label{integral}
\end{align}
Using Theorem \ref{double equidistribution U orbits} with the function $f_{m,i}(u) \coloneqq \phi_i(ms_u)\rho(u)f_i(mu)$ and observing that $||f_{m,i} ||_l \ll ||\phi_i||_l||\rho||_l ||f_i||_l$ and that $||\rho_{|B_r^U}||_l\ll 1 $, it follows that the integral (\ref{integral}) is equal to
\begin{align*}
&\int_M\sum_{i=1}^N \Biggl(\int_U \phi_i(ms_u) f_i(ms_u)\rho(u)d\mu_U(u) \int_{\mathcal{X}}\varphi\int_{\mathcal{X}}\psi \Biggl.\\
& \quad\quad\quad\quad\quad\quad\quad\quad\quad\quad\quad\quad\quad\quad\quad\quad\quad\quad+ \Biggr. O\left( e^{-\delta \min (t_1,t_2, |t_1-t_2|)}||\phi_i||_l||f||_l||\varphi||_l||\psi ||_l\right)\Biggr)d\mu_M(m)\\
&= \int_M\left(\sum_{i=1}^N \int_S \phi_i(ms)f_i(ms)d\mu_S(s)\right)d\mu_M(m) \int_{\mathcal{X}}\varphi\int_{\mathcal{X}}\psi  \\
& \quad\quad\quad\quad\quad\quad\quad\quad\quad\quad\quad\quad\quad\quad\quad\quad\quad\quad+  O\left( Ne^{-\delta \min (t_1,t_2, |t_1-t_2|)}||\phi_i||_l||f||_l||\varphi||_l||\psi ||_l\right)\\
&= \int_K f \int_{\mathcal{X}}\varphi\int_{\mathcal{X}}\psi  + O\left( r^{-\lambda} e^{-\delta \min (t_1,t_2, |t_1-t_2|)}r^{-\nu}||f||_l||\varphi||_l||\psi ||_l\right),
\end{align*}
hence
\begin{align*}
I_{\Lambda,f,\varphi,\psi}(t_1,t_2) =\int_K f \int_{\mathcal{X}}\varphi\int_{\mathcal{X}}\psi  + O\left( \left( r^{-\lambda-\nu} e^{-\delta \min (t_1,t_2, |t_1-t_2|)}+r\right)||f||_l||\varphi||_l||\psi ||_l\right).
\end{align*} 
We take $r=e^{-\gamma \min (t_1,t_2, |t_1-t_2|)}$ with $\gamma = \frac{\delta}{1+\lambda+\nu}$, which yields the claim. 
\end{proof}

	\subsection{Effective pointwise equidistribution along $A$-orbits}

In this subsection we use the quantitative double equidistribution of translated $K$-orbits established in Proposition \ref{double equidistribution K orbits} and previous estimates to derive a quantitative pointwise equidistribution along $\lbrace a_t \rbrace$-orbits of the averages $\frac{1}{N}\sum_{t=0}^{N-1}\hat{\chi}\circ a_t$. The approach is presented in \cite{kleinbock2017pointwise} and works in a more general context to derive an almost-everywhere bound from an L$^p$-bound, $p>1$, using Borel-Cantelli Lemma. We follow the same approach in Lemmas \ref{dyadic estimate}, \ref{lemma 2} and Proposition \ref{pointwise equidistribution} below. 

\begin{prop}
\label{pointwise equidistribution}
Let $(Y,\nu)$ be a probability space, and let $F: Y \times \mathbb{N} \rightarrow \mathbb{R}$ be a measurable function. Suppose there exist constants $ p> 1$ and $C>0$ such that for any integers $0\leq a<b$,
$$
\int_{Y}\left|  \sum_{t=a}^{b-1}F(y,t) \right|^{p}d\nu(y) \leq C (b-a)~.
$$ 
Then, for every $\varepsilon >0$ we have
$$
\sum_{t=0}^{N-1} F(y,t) =O\left(
N^{\frac{1}{p}} \log ^{1+\frac{1}{p}+ \varepsilon} N\right),
$$
for $\nu$-almost every $y \in Y$.
\end{prop}
In the following lemmas, the notations and assumptions are the same as in Proposition \ref{pointwise equidistribution}~. For non-negative integers $m, l$ we write $[m..l)\coloneqq [m,l)\cap \mathbb{N}$.

In order to satisfy in a later application the conditions $t\geq \kappa\log L$ and $t\geq -\frac{1}{\theta}\log L$ from Propositions \ref{non-espace of mass} and \ref{smooth approximation}, we need to use a different dyadic decomposition than the one used in the classical argument. For an integer $s\geq 2$ we consider the following set of dyadic subsets,
\begin{equation*}
L_s\coloneqq \left\lbrace \right[2^i..2^{i+1}\left)~:~ 0\leq i\leq s-2 \right\rbrace \cup \left\lbrace \right[2^ij..2^{i}(j+1)\left)~:~ 2^ij\geq 2^{s-1}, 2^i(j+1)\leq 2^s \right\rbrace \cup \left\lbrace [0..1) \right\rbrace,
\end{equation*}
where the sets of first type $[2^i..2^{i+1})$, $0\leq i\leq s-2$, together with $[0..1)$, are a decompostion of the set $[0..2^{s-1})$ from the classical Schmidt decomposition.\\
For any integer $N\geq 2$ with $2^{s-1}\leq N-1 < 2^s$, the set $[0..N)$ is the disjoint union of at most $2s-1$ subsets in $L_s$ (namely $[0..1)$, the $s-1$ subsets of the first type and at most $s-1$ sets of the second type which can be constructed from the binary expansion of $N-1$). We denote by $L(N)$ this set of subsets, i.e. $[0..N)=\bigsqcup_{I\in L(N)}I$.
\begin{lem} One has
$$\sum_{I\in L_s}\int_{Y}\left|\sum_{t\in I} F(y,t) \right|^p d\nu(y) \leq Cs2^{s}. 
$$
\label{dyadic estimate}
\end{lem}
\begin{proof} Since $L_s$ is a subset of the set of all dyadic sets  $[2^ij..2^i(j+1))$ where $i,j$ are non-negative integers and $2^i(j+1)\leq 2^s$, we have 
\begin{align*}
\sum_{I\in L_s}\int_{Y}\left|\sum_{t\in I} F(y,t) \right|^p d\nu(y)  &\leq \sum_{i=0}^{s-1}\sum_{j=0}^{2^{s-i}-1} \int_{Y}\left|\sum_{t\in I} F(y,t) \right|^p d\nu(y) \\
&\leq \sum_{i=0}^{s-1}\sum_{j=0}^{2^{s-i}-1} C2^i\\
&\leq Cs2^s.
\end{align*}
\end{proof}
\begin{lem}
For every $\varepsilon >0$, there exists a sequence of measurable subsets $\lbrace Y_{s}\rbrace_{s \in \mathbb{N}}$ of $Y$ such that:
\begin{enumerate}
\item $\nu(Y_s) \leq Cs^{-(1+p\varepsilon)}.$
\item For every integer $N\geq 2$ with $2^{s-1}\leq N-1<2^s$ and every $y \notin Y_s$ one has
\begin{equation}
\left| \sum_{t=0}^{N-1}F(y,t) \right| \ll 
2^{\frac{s}{p}}s^{1+\frac{1}{p}+\varepsilon}.
\end{equation}
\end{enumerate}
\label{lemma 2}
\end{lem}
\begin{proof}
Consider 
\begin{equation*}
Y_{s}=
\left\lbrace y \in Y~:~ \sum_{I\in L_s} \left| \sum_{t\in I}F(y,t)\right|^p ~>~2^{s} s^{2+ p\varepsilon} \right\rbrace.
\end{equation*}
The first assertion follows from Lemma \ref{dyadic estimate} and Markov's Inequality.\\
Further, for $2^{s-1}\leq N-1<2^s$ and $y\notin Y_s$, and using that $[0..N)=\bigsqcup_{I\in L(N)}I$ with $L(N)$ of cardinality at most $2s-1$, we have
\begin{align*}
\left|\sum_{t=0}^{N-1} F(y,t)\right|^p &= \left|\sum_{I\in L(N)}\sum_{t\in I} F(y,t)\right|^p\\
&\leq (2s-1)^{p-1}\sum_{I\in L(N)}\left|\sum_{t\in I} F(y,t)\right|^p &\text{(by Hölder's Inequality)}\\
&\leq (2s-1)^{p-1}\sum_{I\in L_s}\left|\sum_{t\in I} F(y,t)\right|^p \\
&\ll_p s^{1+p+p\varepsilon}2^{s} &\text{(since }y\notin Y_{s})
\end{align*}
which yields the claim by raising to the power $\frac{1}{p}$.
\end{proof}
\begin{proof}[Proof of Proposition \ref{pointwise equidistribution} (see \cite{kleinbock2017pointwise})]
Let $\varepsilon>0$ and choose a sequence of measurable subsets $\{Y_s\}_{s\in \mathbb{N}}$ as in Lemma \ref{lemma 2} . Observe that
$$\sum_{s=1}^\infty \nu(Y_s) \leq \sum_{s=1}^\infty C s^{-(1+p\varepsilon)} < \infty.
$$
The Borel-Cantelli lemma implies that there exists a full-measure subset $Y(\varepsilon)\subset Y$ such that for every $y \in Y(\varepsilon) $ there exists $s_y \in \mathbb{N}$ such that for all $s \geq s_y$ we have $y \notin Y_s$.\\
Let $N\geq 2$ and $s= 1 + \left\lfloor \log N-1 \right\rfloor$, so that $2^{s-1}\leq N-1 < 2^s $. Then, for $N-1 \geq 2^{s_y}$ we have $s> s_y$ and $y \notin Y_s$, thus
\begin{align*}
\left|\sum_{t=0}^{N-1}F(y,t) \right| &\ll 2^{s\frac{1}{p}}s^{1+\frac{1}{p} +\varepsilon}
\\
&\leq
(2N)^{\frac{1}{p}}\log^{1+\frac{1}{p}+ \varepsilon}(2N).
\end{align*}
This implies the claim for $y \in \cap_{N\in \mathbb{N}}Y(1/N)$.
\end{proof}

We now apply Proposition \ref{pointwise equidistribution} to the counting function $\sum_{t}\hat{\chi}\circ a_t$. We denote by vol$(F_{1,c})$ the average of the Siegel transform from Proposition \ref{siegel mean value thrm} for the function $\chi=\chi_{F_{1,c}}$,
$$\text{vol}(F_{1,c}) \coloneqq \int_{\mathcal{C}} \chi(x)dx = \int_{\mathcal{X}} \hat{\chi}(\Lambda)d\mu_\mathcal{X}(\Lambda).$$
\begin{thm}
\label{pointwise equidistribution counting function}
There exists $\delta<1$ depending only on the dimension $n$ such that for almost every $k \in K$ we have
$$
 \sum_{t=0}^{N-1} \hat{\chi}(a_t k \Lambda_0) = N\,\emph{vol}(F_{1,c}) + O_{k} (N^{\delta}).
$$
\end{thm}

\begin{proof}
Using Proposition \ref{pointwise equidistribution} it is enough to show that there exists $p> 1$ such that for every set $[a..b)$ in $L_s$ from lemmas \ref{dyadic estimate} and \ref{lemma 2}, we have
\begin{equation}
\label{L2-delta estimate}
\left|\left| \sum_{t=a}^{b-1} \left( \hat{\chi}-\int_{\mathcal{X}}\hat{\chi}\right)\circ a_t \right|\right|_{L^{p}(\mathcal{Y})}^{p} \ll (b-a)~.
\end{equation}
Using the estimates for the truncated Siegel transform from Proposition \ref{bounds truncated siegel transform}, we have for all $1< p<2$, $\frac{2p}{3p-2}\leq \tau < n$ and $t\geq \kappa \log L$,
\begin{align}
\left|\left|  ( \hat{\chi}\circ a_t-\int_{\mathcal{X}}\hat{\chi})- ( \hat{\chi}^{(L)}\circ a_t-\int_{\mathcal{X}}\hat{\chi}^{(L)})\right|\right|_{L^{p}_\mathcal{Y}} &\leq \left|\left|   \hat{\chi}\circ a_t- \hat{\chi}^{(L)}\circ a_t\right|\right|_{L^{p}_\mathcal{Y}}+ \int_{\mathcal{X}}\left|\hat{\chi}-\hat{\chi}^{(L)} \right| \nonumber\\
& \ll L^{-\frac{\tau(2-p)}{2p}}~+~L^{-(\tau-1)} \nonumber\\
& \ll L^{-\frac{\tau(2-p)}{2p}}.
\label{estimate truncation}
\end{align}
Further, using Proposition \ref{smooth approximation} and the estimates from Proposition \ref{bounds truncated siegel transform} and (\ref{epsilon approximation of chi}), we have for all $t\geq -\frac{1}{\theta}\log \varepsilon$,
\begin{align}
&\left\lVert  ( \hat{\chi}^{(L)}\circ a_t-\int_{\mathcal{X}}\hat{\chi}^{(L)}) - ( \hat{f}_{\varepsilon}^{(L)}\circ a_t-\int_{\mathcal{X}}\hat{f}_{\varepsilon}^{(L)})\right\rVert_{L^{p}_\mathcal{Y}}  \leq \left|\left|   \hat{\chi}^{(L)}\circ a_t- \hat{f}_{\varepsilon}^{(L)}\circ a_t\right|\right|_{L^{p}_\mathcal{Y}}+ \int_{\mathcal{X}}\left|\hat{\chi}^{(L)}-\hat{f}_{\varepsilon}^{(L)} \right| \nonumber\\
&\qquad \qquad \qquad\leq \left|\left|(\widehat{\chi -f_{\varepsilon}})^{(L)}\circ a_t \right|\right|_{\infty}^{\frac{p-1}{p}}\cdot \left|\left|(\widehat{\chi -f_{\varepsilon}})^{(L)}\circ a_t \right|\right|_{L^1_\mathcal{Y}}^{\frac{1}{p}}+ \int_{\mathcal{C}}\left|\chi-f_{\varepsilon} \right| \nonumber \\
&\qquad \qquad \qquad \ll L^{\frac{p-1}{p}}~\varepsilon^{\frac{1}{p}}~ +~ \varepsilon \nonumber \\
& \qquad \qquad \qquad\ll L^{\frac{p-1}{p}}~\varepsilon^{\frac{1}{p}}.
\label{estimate smooth}
\end{align}
Further, using the effective double equidistribution for smooth compactly supported functions from Proposition \ref{double equidistribution K orbits} and the estimates for the $C^l$-norm in (\ref{bound for Cl norm truncated}) and (\ref{epsilon approximation of chi}), we have 
\begin{align}
\left\lVert  \sum_{t=a}^{b-1}( \hat{f}_{\varepsilon}^{(L)}-\int_{\mathcal{X}}\hat{f}_{\varepsilon}^{(L)})\circ a_t \right\rVert_{L^{p}_\mathcal{Y}} &\leq \left\lVert  \sum_{t=a}^{b-1}( \hat{f}_{\varepsilon}^{(L)}-\int_{\mathcal{X}}\hat{f}_{\varepsilon}^{(L)})\circ a_t \right\rVert_{L^{2}_\mathcal{Y}} \nonumber \\
&= \left(\sum_{t_1,t_2=a}^{b-1} \int_{\mathcal{Y}}(\hat{f}_{\varepsilon}^{(L)}-\int_{\mathcal{X}}\hat{f}_{\varepsilon}^{(L)})\circ a_{t_1}\cdot (\hat{f}_{\varepsilon}^{(L)}-\int_{\mathcal{X}}\hat{f}_{\varepsilon}^{(L)})\circ a_{t_2}~d\mu_\mathcal{Y} \right)^{1/2} \nonumber \\
&\ll \left(\sum_{t_1,t_2=a}^{b-1} \left\lVert \hat{f}_{\varepsilon}^{(L)}\right\rVert_l^2 ~e^{-\gamma \min (t_1, t_2, |t_1-t_2|)} \right)^{1/2} \nonumber \\
&\ll \left\lVert \hat{f}_{\varepsilon}^{(L)}\right\rVert_l ~ (b-a)^{1/2} \quad \ll L~\varepsilon^{-l}(b-a)^{1/2},
\label{estimate equidistribution}
\end{align}
where we used the estimates
\begin{align*}
&\quad\quad\quad\sum_{\substack{t_1,t_2=a\\t_1\leq\min(t_2,|t_1-t_2|)}}^{b-1}e^{-\gamma \min (t_1, t_2, |t_1-t_2|)}\leq \sum_{t_1,t_2=a}^{b-1}e^{-\gamma t_1} \leq (b-a)\sum_{t_1=a}^{b-1}e^{-\gamma t_1} \ll b-a,\\
&\text{similarly }\quad \sum_{\substack{t_1,t_2=a\\t_2\leq\min(t_1,|t_1-t_2|)}}^{b-1}e^{-\gamma \min (t_1, t_2, |t_1-t_2|)}\ll b-a,\\
&\text{and } \sum_{\substack{t_1,t_2=a\\|t_1-t_2|\leq\min(t_1,t_2)}}^{b-1}e^{-\gamma \min (t_1, t_2, |t_1-t_2|)}\leq\sum_{\substack{l=a-b}}^{b-a}\sum_{\substack{t_1=a\\~~~ t_1\in [a+l,b-1+l]}}^{b-1}e^{-\gamma |l|}\leq \sum_{l=a-b}^{b-a}(b-a-|l|)e^{-\gamma |l|}\ll b-a.
\end{align*}
Combining (\ref{estimate truncation}), (\ref{estimate smooth}) and (\ref{estimate equidistribution}) we obtain
\begin{align}
\label{altogether bound}
\left|\left| \sum_{t=a}^{b-1} \left( \hat{\chi}-\int_{\mathcal{X}}\hat{\chi}\right)\circ a_t \right|\right|_{L^{p}_\mathcal{Y}} &\ll  (b-a)L^{-\frac{\tau(2-p)}{2p}}~+~(b-a)L^{\frac{p-1}{p}}\varepsilon^{\frac{1}{p}}~+~(b-a)^{1/2}L\varepsilon^{-l}
\end{align}
Setting the summands in (\ref{altogether bound}) to be equal, we have an optimal bound for
\begin{equation}
\label{choice L and epsilon}
\varepsilon = L^{1-p-\frac{\tau(2-p)}{2}} \quad \text{and}\quad L = (b-a)^{\frac{p}{2+(1+pl)(\tau(2-p)+2(p-1))}}.
\end{equation}
We verify that for all but finitely many sets $[a..b)$ in $L_s$, we have that for all $t \in [a..b)$ the conditions $t\geq \kappa \log  L$ and $t\geq -\frac{1}{\theta}\log \varepsilon$ are satisfied. We write $L =(b-a)^\beta$ with $\beta >0$. For sets $[a..b)$ of the first type, i.e. $[a..b)=[2^i,2^{i+1})$ with $0\leq i\leq s-2$, we have
\begin{align*}
\log L = \log &\left( 2^{i\beta} \right) = i \log(2^\beta) \\
&\leq \frac{1}{\kappa}2^i \quad \text{for all } C_1\leq i \leq s-2,\text{ for some fixed integer }C_1 \text{ and large enough }s,\\
&\leq \frac{1}{\kappa} t \quad \text{ for all } t \in [2^i..2^{i+1}) \text{ for all }  C_1\leq i\leq s-2.
\end{align*}
For sets $[a..b)$ of the second type, i.e. $[a..b)=[2^ij,2^{i}(j+1))$ with $2^ij\geq 2^{s-1}$ and $2^i(j+1)\leq 2^{s}$, the condition is then a fortiori satisfied, since those sets have all a larger lower bound and are all at most as large as the set of first type $[2^{s-2}, 2^{s-1})$.\\
One verifies similarly that there exists a constant $C_2>0$ such that 
\begin{align*}
\quad \log \varepsilon &\geq -\theta t \quad \text{ for } t \text{ in all sets }[a..b)\in L_s \text{ but those }[2^i..2^{i+1}) \text{ with } 0\leq i< C_2.
\end{align*}
Further, since $\sup_{t\geq 0}||\hat{\chi}\circ a_t ||_{L^p_\mathcal{Y}}<\infty$, we have for all sets $[a..b)=[2^i..2^{i+1})$ with $i <\max(C_1,C_2)$
$$
\left|\left| \sum_{t=a}^{b-1} \left( \hat{\chi}-\int_{\mathcal{X}}\hat{\chi}\right)\circ a_t \right|\right|_{L^{p}_\mathcal{Y}} \ll (b-a) \left|\left|  \left( \hat{\chi}-\int_{\mathcal{X}}\hat{\chi}\right)\circ a_t \right|\right|_{L^{p}_\mathcal{Y}} = O(1).
$$
Hence, (\ref{altogether bound}) and (\ref{choice L and epsilon}) yield, for all sets $[a..b)$ in $L_s$,
\begin{align*}
\left|\left| \sum_{t=a}^{b-1} \left( \hat{\chi}-\int_{\mathcal{X}}\hat{\chi}\right)\circ a_t \right|\right|_{L^{p}_\mathcal{Y}} &\ll (b-a)^{\delta_p}.
\end{align*}
with the exponent $\delta_p$ given by 
$$
\delta_p = 1-\frac{\tau(2-p)}{4+ 2(pl+1)(\tau(2-p)+2(p-1))}~.
$$
We observe that $\delta_p$ is an increasing function of $p$ for $1< p<2$, with $\lim_{p\rightarrow 1}\delta_p = 1-\frac{\tau}{4+2\tau(l+1)}<1$ and $\lim_{p\rightarrow 2}\delta_p=1$
hence, by the intermediate value theorem, there exists $1<p<2$ with $ \delta_p=\frac{\,1\,}{p}<1$. We denote by $\delta$ this exponent and note that it depends only on the dimension $n$.
\end{proof}

\section{Effective estimate for the counting function}	

\begin{proof}[Proof of Theorem \ref{main theorem}]
\label{proof of main theorem}
Using Theorem \ref{pointwise equidistribution counting function} with the estimate (\ref{resandwiching N_T,c})  we have
\begin{equation*}
\left(\left\lfloor T-r_0 \right\rfloor -1\right)\text{vol}(F_{1,c_\mathcal{l}})+O\left(T ^{\delta} \right)+O\left(\mathcal{l}^n \right)\leq  N_{T,c}(\alpha) +O(1) \leq \left(\left\lfloor T+r_0 \right\rfloor\right)\text{vol}(F_{1,c})+O\left(T ^{\delta}\right).
\end{equation*}
Using further (\ref{volum F_1,c approximation}) with $\mathcal{l}=\left\lfloor T^{\frac{1}{n+1}} \right\rfloor $, we have
\begin{align*}
T\text{vol}(F_{1,c_\mathcal{l}})+O\left(T ^\delta\right) + O\left(\mathcal{l}^n \right)
&= T \left( \text{vol}(F_{1,c}) + O(\mathcal{l}^{-1})\right) +O\left(T ^\delta\right)+ O\left(\mathcal{l}^n \right)\\
&= T  \text{vol}(F_{1,c}) + O\left(T^{\gamma} \right),\quad \text{with }\gamma=\max \left(\frac{n}{n+1},\delta\right),
\end{align*}
hence
\begin{equation*}
N_{T,c}(\alpha)
=T\text{vol}(F_{1,c})+O\left(T ^{\gamma}\right).
\end{equation*}
Since full-measure sets in S$^n$ correspond to full-measure sets in $K$, we conclude that this last estimate holds for almost every $\alpha \in$ S$^n$.
\end{proof}

\newpage

\printbibliography

\end{document}